\documentclass[11pt,reqno]{amsart}
\usepackage{a4wide}
\usepackage[utf8x]{inputenc}
\usepackage[T1]{fontenc}
\usepackage[english]{babel}
\usepackage{hyperref}
\usepackage{amsfonts,amsmath,amssymb,amsthm,amsrefs}
\usepackage{latexsym}
\usepackage{layout}
\usepackage{dsfont}
\usepackage{color}
\usepackage[pdftex]{graphicx}

\usepackage{ifpdf}
\usepackage{color}
\definecolor{webgreen}{rgb}{0,.5,0}
\definecolor{webbrown}{rgb}{.8,0,0}
\definecolor{emphcolor}{rgb}{0.5,0.95,0.95}

\newtheorem{theorem}{Theorem}[section]
\newtheorem{proposition}{Proposition}[section]
\newtheorem{lemma}{Lemma}[section]
\newtheorem{remark}{Remark}[section]
\newtheorem{definition}{Definition}[section]
\newtheorem{assumption}{Assumption}[section]
\newtheorem{corollary}{Corollary}[section]
\newcommand{\p}{\mathbb{P}}
\newcommand{\e}{\mathbb{E}}

\newcommand{\reals}{\mathbb{R}}

\newcommand{\ind}{\mathbf{1}}
\newcommand{\Ind}{\mathbb{I}}

\newcommand{\diff}{\mathrm{d}}

\newcommand{\cH}{\mathcal{H}}

\newcommand{\homega}{\omega_q}

\newcommand{\ruintime}{\sigma^\pi_\omega}

\allowdisplaybreaks

\newcommand{\BRA}[1]{{{\left\{#1\right\}}}} 
\newcommand{\PAR}[1]{{{\left(#1\right)}}} 
\newcommand{\SBRA}[1]{{{\left[#1\right]}}} 
\renewcommand{\leq}{\leqslant}
\renewcommand{\geq}{\geqslant}

\begin{document}

\title[Impulse control with a level-dependent intensity of ruin]{Impulse control in a spectrally negative L\'evy model with a level-dependent intensity of bankruptcy}

\author[D. Mata]{Dante Mata}

\address{D\'epartement de math\'ematiques, Universit\'e du Qu\'ebec \`a Montr\'eal (UQAM), 201 av.\ Pr\'esident-Kennedy, Montr\'eal (Qu\'ebec) H2X 3Y7, Canada}
\email{mata\_lopez.dante@uqam.ca}


\date{\today}


\begin{abstract}
We consider an optimal dividend problem with transaction costs where the surplus is modelled by a spectrally negative L\'evy process in an Omega model. n this model, the surplus is allowed to spend time below the critical ruin level, but is penalised by a state-dependent intensity of bankruptcy. We show that under the spectrally negative model an optimal strategy is such that the surplus is reduced to a level $c_1$ whenever they are above another level $c_2$, and that such levels are unique under the additional assumption that the L\'evy measure has a log-convex tail. We describe a numerical method to compute the optimal values $c_1$ and $c_2$.
\end{abstract}

\maketitle


\section{Introduction}
In this paper, we study a version of De Finetti's optimal dividend problem in continuous time in which a transaction cost is incurred each time a dividend payment is made. The presence of this fixed transaction cost prevents the controller from paying dividends continuously in time, and therefore only lump sum dividend payments are feasible. This class of problems is widely known in the literature as the impulse comtrol problem.\\
We assume that the underlying dynamics of the risk process is described by a spectrally negative L\'evy process (SNLP), i.e., a process that exhibits only downwards jumps. This class of processes has been widely used in the insurance literature, as they are a generalisation of the classical Cram\'er-Lundberg process. Early instances of the impulse control problem can be found in the seminal works of Jeanblanc-Picqu\'e and Shiryaev \cite{JS1995} and Paulsen \cite{Paulsen2007} under the diffusive model; and afterwards  Loeffen \cite{loeffen2009} studied the impulse control problem under the SNLP model. Recently, Moreno-Franco and P\'erez \cite{MP2024} have studied a version of the dividend problem with transaction costs driven by an SNLP with Poissonian decision times. In these preliminary studies, it was shown that a double barrier strategy is optimal. In broad terms, a double barrier strategy requires two levels $(c_1,c_2)$, with $c_1 < c_2$ such that whenever the surplus is above the level $c_2$, then a lump sum is performed to bring the reserves down to level $c_1$. It must be remarked that each lump sum payment must cover at least the transaction cost.\\
These previous studies, namely \cite{JS1995,loeffen2009,Paulsen2007} rely on the concept of \textit{classical} ruin, which is defined as the first passage time of the controlled surplus below  the critical level 0. Recently, this assumption has been generalised to consider more realistic models of ruin. For instance, the concept of Parisian ruin allows the surplus to remain below the critical level, however ruin is declared if the process remains negative after a certain amount of time, we refer the reader to \cite{renaud2019} for a dividend problem under the constraint of Parisian ruin. An impulse control problem under the constraint of Parisian ruin has been studied by Czarna and Kaszubowski \cite{CK2020} under the \textit{refracted} spectrally negative model, and more recently Wang et al. \cite{Wang2024} have studied this problem under the Chapter 11 bankruptcy constraint. In this paper, we consider a more general version of Parisian ruin, the so-called \textit{Omega}-model, where a bankruptcy rate function $\omega$ is provided and if the surplus is at level $x < 0$, then the rate of bankruptcy is given by $\omega(x)$. This notion of bankruptcy was defined rigorously for SNLP's by Li and Palmowski \cite{li-palmowski_2018}, where they developed a suite of fluctuation identities expressed in terms of the \textit{generalised} scale functions, which are defined as the solution of an integral equation. Recently, Mata and Renaud \cite{mata-renaud_2024} studied the classical De Finetti dividend problem under this model of ruin, and derived analytical properties of the generalised scale functions under the assumption that the L\'evy measure has a log-convex tail. In the present paper we use both the fluctuation identities found in \cite{li-zhou_2014} and the analytical properties of the generalised scale function from \cite{mata-renaud_2024}, to prove the optimality of a double barrier strategy and, moreover, prove the uniqueness of an optimal strategy when the tail of the L\'evy measure is log-convex.\\
This paper is structured as follows, in Section \ref{Sec:Preliminary} we introduce the underlying surplus process, the so-called scale functions and the class of generalised scale functions, as well as the formulation of the impulse control problem. In Section \ref{Sec:Barrier} we introduce the double barrier strategies, compute their performance function in term of the generalised scale functions, and propose candidate values for the optimal strategy; whereas in Section \ref{Sec:Verification} we state a verification lemma and provide rigorous proofs to show that a double barrier strategy is optimal. In Section \ref{Sec:Optim:Convex} we study the uniqueness of an optimal strategy under the assumption that the tail of the L\'evy measure is log-convex, and we propose a method to compute the optimal barriers. Finally, in Section \ref{Sec:Numerical} we present numerical computations to illustrate our theoretical results. 


\section{Preliminaries}\label{Sec:Preliminary}
\subsection{L\'evy Processes}
Consider a filtered probability space $\left( \Omega, \mathcal{F}, \left\lbrace \mathcal{F}_t, t \geq 0 \right\rbrace, \p \right)$ such that $\mathcal{F}_\infty := \bigvee_{t \geq 0} \mathcal{F}_t \subsetneq \mathcal{F}$, where $\bigvee_{t \geq 0} \mathcal{F}_t$ is the smallest $\sigma$-algebra containing $\mathcal{F}_t$ for all $t \geq 0$. Let $X=\left\lbrace X_t , t \geq 0 \right\rbrace$ be a spectrally negative Lévy process (adapted to the filtration $\left\lbrace \mathcal{F}_t, t \geq 0 \right\rbrace$) with Laplace exponent $\theta \mapsto \psi (\theta)$ given by the L\'evy-Khintchine formula
\begin{equation}\label{eq.Laplace}
\psi(\theta) = \gamma \theta + \frac{1}{2} \sigma^2 \theta^2 + \int^{\infty}_0 \left( \mathrm{e}^{-\theta z} - 1 + \theta z \ind_{(0,1]}(z) \right) \nu(\mathrm{d}z) ,
\end{equation}
where $\gamma \in \reals$ and $\sigma \geq 0$, and where $\nu$ is a $\sigma$-finite measure on $(0,\infty)$, called the L\'{e}vy measure of $X$, satisfying
\begin{equation*}
\int^{\infty}_0 (1 \wedge x^2) \nu(\mathrm{d}x) < \infty .
\end{equation*}
For more details on spectrally negative Lévy processes, see e.g.
\cite{kyprianou_2014}.\\
In the rest of this paper, we make the following standing assumption:
\begin{assumption}\label{standing-assumption}
	\begin{enumerate}
	\item If the paths of $X$ are of bounded variation (BV), then the Lévy measure is assumed to be absolutely continuous with respect to Lebesgue measure.
	\item The tail of the L\'evy measure, i.e., the mapping $x \mapsto \nu(x,\infty)$ is assumed to be log-convex.
	\end{enumerate}
\end{assumption}

In what follows, we will use the following notation: the law of $X$ when starting from $X_0 = x$ is denoted by $\p_x$ and the corresponding expectation by $\e_x$. We write $\p$ and $\e$ when $x=0$.

\subsection{Optimisation problem}\label{sect:optimisation-problem}

Let $X$ be the underlying surplus process. An impulse control strategy $\pi$ is represented by a non-decreasing, right-continuous and adapted stochastic process $L^\pi = \left\lbrace L^\pi_t , t \geq 0 \right\rbrace$ such that $L^\pi_{0-} = 0$, where $L^\pi_t$ represents the cumulative amount of dividends paid up to time $t$ under this control strategy. In addition, we assume that $L^\pi$ is a pure-jump process, i.e.
\[
L^\pi_t = \sum_{0 \leq s < t} \Delta L^\pi_s, \quad t \geq 0.
\]
Throughout this paper, we mean $\Delta L^\pi_s = L^\pi_{s+} - L^\pi_s$.\\
For a given strategy $\pi$, the corresponding controlled process $U^\pi = \left\lbrace U^\pi_t , t \geq 0 \right\rbrace$ is defined by $U^\pi_t = X_t - L^\pi_t$. A strategy is said to be admissible if $U^\pi_t - \Delta L^\pi_t \geq 0$ for all $t$. Let $\Pi$ be the set of admissible strategies.\\
In order to define the bankruptcy time of interest, let us fix a \textit{bankruptcy rate function} $\omega \colon \reals \to [0,\infty)$ (see Definition~\ref{def:bankruptcy-rate-function} below), that is roughly speaking a non-negative and non-increasing function having support in $(-\infty,0]$. We define the following termination time associated with $\omega$: for $\pi \in \Pi$, set
\begin{equation}\label{omega-ruin}
\ruintime = \inf \left\lbrace t>0 \colon \int_0^t \omega(U^\pi_s) \mathrm{d}s > \mathbf{e}_1 \right\rbrace ,
\end{equation}
where $\mathbf{e}_1$ is an exponential random variable with unit mean, independent of the surplus.

Finally, fix a time-preference rate $q > 0$ and a cost per transaction $\beta > 0$. The performance function associated to an admissible strategy $\pi \in \Pi$ is defined by
\[
v_\pi (x) = \e_x \left[ \int_0^{\ruintime} \mathrm{e}^{-q t} \diff \PAR{L^\pi_t - \sum_{0 \leq s < t}\beta \Ind_{\BRA{\Delta L^\pi_s > 0}} } \right] , \quad x \in \reals .
\]
Note that due to the pure jump nature of the admissible controls, we wan rewrite the performance function as
\[
v_\pi (x) = \e_x\SBRA{\sum_{0 \leq t < \ruintime} \mathrm{e}^{-q t} \PAR{\Delta L^\pi_t - \beta \Ind_{\BRA{\Delta L^\pi_t > 0}} } }.
\]
At this point, we find relevant to present the following alternate representation of the performance function of an admissible strategy, which aloows us to interpret our problem as an infinite-horizon dividend problem where dividend payments are penalised by a level-dependent discounting. We omit the proof, as it is similar to that of Lemma 3.1 in \cite{mata-renaud_2024}.
\begin{lemma}\label{lem:alternate}
	For any $\pi \in \Pi$ we have
	\[
	v_\pi(x) = \e_x\SBRA{ \sum_{0 \leq t < \infty} \mathrm{e}^{-q t} \PAR{\Delta L^\pi_t - \beta \Ind_{\BRA{\Delta L^\pi_t > 0}} }}, \quad x \in \reals.
	\]
\end{lemma}
Our goal is to compute
\begin{equation}\label{eq.Value}
v_\ast (x) = \sup_{\pi \in \Pi} v_\pi (x) , \quad x \in \reals ,
\end{equation}
the value function of this control problem by finding an optimal strategy $\pi_\ast \in \Pi$, i.e., an admissible strategy such that
\[
v_{\pi_\ast} (x) = v_\ast (x) , \quad x \in \reals .
\]

\subsection{Bankruptcy rate function}

We provide a proper definition of a \textit{bankruptcy rate function}, a name that has been used already in \cite{gerber-et-al_2012}.

\begin{definition}\label{def:bankruptcy-rate-function}
	We say that $\omega \colon \reals \to [0,\infty)$ is a discounting intensity if:
	\begin{enumerate}
		\item it is non-increasing;
		\item it is \textit{ultimately constant}, i.e., there exist $a \in (-\infty,0]$ and $\phi > 0$ for which $\omega(x)=\phi$ for all $x \in (-\infty,a)$, and there exists $\rho \in [0, \omega(0-)]$ such that $\omega(x)=\rho$ for all $x \in [0,\infty)$;
		\item it is piecewise continuous, i.e., there exists a finite partition $\{a_1, a_2, \dots, a_n, a_{n+1}\}$ of $[a,0]$ given by $a = a_1 < a_2 < \cdots < a_n \leq a_{n+1} = 0$ and such that:
		\begin{enumerate}
			\item $\omega$ is continuous on each sub-interval $(a_k,a_{k+1})$ for $k \in \lbrace 1,2,\cdots,n \rbrace$;
			\item if $n \geq 2$, then $\omega(a_k-)>\omega(a_k+)$ for each $k \in \lbrace 2,\cdots,n \rbrace$.
		\end{enumerate}
	\end{enumerate}
\end{definition}
\begin{remark}
	In our model, we say that $\omega$ is a bankruptcy rate function if it is a discounting intensity with $\rho=0$.
\end{remark}

As opposed to other elements in the partition, a discounting intensity might or might not have jumps at $a$ and $0$. More precisely, we have $\phi=\omega(a-) \geq \omega(a+)$ and $\omega(0-) \geq \rho=\omega(0)$. Thus, if $\omega$ is continuous on $(-\infty,0)$, then $n=1$ and $a \leq 0$ and thus the only possible discontinuity point is now at $0$. This is the case for example if $\omega(x) = \phi \ind_{(-\infty,0)} (x) + \rho \ind_{[0,\infty)}(x)$. However, it is possible for $\omega$ to be continuous on $\reals$. For example,
\[
\omega(x) = \phi \ind_{(-\infty,-\phi)}(x) - x \ind_{[-\phi,0)}(x)
\]
is a continuous bankruptcy rate function.

\begin{remark}
A piecewise constant bankruptcy rate function yields a termination time that is already significantly more general than a Parisian ruin time with exponential delays. Indeed, Parisian ruin with exponential rate $\phi>0$ corresponds to the bankruptcy rate function $\omega(x) = \phi \ind_{(-\infty,0)}(x)$.
\end{remark}

%
%
\section{Scale functions}\label{sect:scale}
For each $q \geq 0$, there exists a function $W_q$ known as the \textit{scale function}, such that $W_q(x) = 0$ for all $x\in (-\infty,0)$, and it is strictly increasing on $(0,\infty)$. In addition, it is characterised by its Laplace transform in the following way
\begin{equation}\label{scale:function:def}
	\int_0^\infty \mathrm{e}^{-\theta x} W_q (x) \mathrm{d}x = \frac{1}{\psi(\theta)-q} , \quad \text{for all $\theta> \Phi(q),$}
\end{equation}
where $\Phi(q)\sup \left\lbrace \lambda \geq 0 \colon \psi(\lambda)=p \right\rbrace$. We refer the reader to \cite{kuznetsov-et-al_2012} for a detailed analysis on the theory of scale functions.\\

As a consequence of item 1 from Assumption \ref{standing-assumption}, we have that the scale function $W_q$ is continuously differentiable on $(0,\infty)$, for any $q \geq 0$; see, e.g., \cite{kuznetsov-et-al_2012}.

Recall that $q$ is the discount rate. Since \cites{avram-et-al_2007,loeffen_2008}, it has been widely known that the $q$-scale function $W_q$ is an essential object in the study of the optimal dividends problem with classical ruin. When exponential Parisian ruin is considered, as in \cite{renaud2019}, a second family of scale functions is needed. If Parisian ruin occurs at rate $\phi>0$, i.e., if $\omega(x) = \phi \ind_{(-\infty,0)}(x)$, then $Z_q (\cdot;\Phi(\phi+q))$ is the scale function of interest. In general, for fixed $0<q<s$, define, for each $x \in \reals$,
\begin{align}
	Z_q (x;\Phi(s)) &= \mathrm{e}^{\Phi(s)x} \left(1-(s-q) \int_0^x \mathrm{e}^{-\Phi(s)y} W_{q}(y) \mathrm{d}y \right)\label{eq:Zq-def} \\
	&= (s-q) \int_0^\infty \mathrm{e}^{-\Phi(s)y} W_{q}(x+y) \mathrm{d}y ,\notag
\end{align}
where for $x \leq 0$, we have $Z_q (x;\Phi(s))=\mathrm{e}^{\Phi(s) x}$.  In particular, we have $Z_q (x;\Phi(\phi+q))= \phi \int_0^\infty \mathrm{e}^{-\Phi(\phi+q)y} W_q(x+y) \mathrm{d}y$.

%
%
%

\subsection{Omega scale functions}\label{Sec:Solution:Complete}

For the control problem under study, another \textit{scale function}, which is a generalization of both $W_{q}(\cdot)$ and $Z_q (\cdot;\Phi(\phi+q))$, will be needed. A theory of scale functions based on a bankruptcy rate function $\omega$ (also called a \textit{killing intensity}) has been developed in \cite{li-palmowski_2018}. In that paper, the $\omega$-scale function $\cH^{\omega} \colon \reals \to \reals$ is defined as the solution of the following functional equation:
\begin{equation}\label{eq:functional-eq}
	\cH^{\omega}(x) = \mathrm e^{\Phi(\phi)(x-a)} + \int_{a}^x W_{\phi}(x-y) (\omega(y) - \phi) \cH^{\omega}(y) \diff y , \quad x \in \reals .
\end{equation}
See Lemma 2.1 and Section 2.4 in \cite{li-palmowski_2018} for more details.

In the above definition and in what follows, it is assumed that, for any integrable function $f$, if $a \geq b$ then $\int_a^b f(x) \diff x = 0$. In general, by $\int_a^b$ we mean $\int_{[a,b)}$.

\begin{remark}
	Note that~\eqref{eq:functional-eq} does not appear verbatim in \cite{li-palmowski_2018}. We have performed a translation in Equations~(2.22)-(2.23) of \cite{li-palmowski_2018} to obtain~\eqref{eq:functional-eq}. This modification is needed to be in agreement with our definition of a bankruptcy rate function.
\end{remark}

Below we list some key results regarding the analytic properties of the $\omega$-scale function $\cH^{\omega}$. We refer the reader to \cite{mata-renaud_2024} for detailed proofs.

\begin{theorem}[Theorem 2.1 in \cite{mata-renaud_2024}]\label{thm:solution-integral-eq}
	Under Assumption~\ref{standing-assumption}, the $\omega$-scale function $\cH^{\omega}$ has the following differentiability properties:
	\begin{enumerate}
		\item If $X$ has paths of UBV, then $\cH^{\omega}$ is continuously differentiable on $\reals$. In addition, if $\sigma >0$, then $\cH^{\omega}$ is twice continuously differentiable on $\reals \setminus \{a_1,a_2,\cdots,a_n\}$.
		\item If $X$ has paths of BV, then $\cH^{\omega}$ is:
		\begin{enumerate}
			\item continuously differentiable on $\reals \setminus \{a_1, a_2, \dots, a_n\}$ if $\omega(0-)=0$;
			\item continuously differentiable on $\reals \setminus \{a_1, a_2, \dots, a_n,a_{n+1}\}$ if $\omega(0-)>0$.
		\end{enumerate}
	\end{enumerate}
\end{theorem}

\begin{proposition}[Proposition 2.1 in \cite{mata-renaud_2024}]\label{prop:log-convex-general}
	If the function $x \mapsto \nu(x,\infty)$ is log-convex on $(0,\infty)$, then $\mathcal{H}^{\omega \prime}$ is log-convex on $(0,\infty)$. In particular, it is also convex on $(0,\infty)$.
\end{proposition}

\section{Double barrier strategies}\label{Sec:Barrier}
An important type of strategy in impulse control models is the so-called $(c_1,c_2)$ strategy. This strategy is defined such that whenever the surplus process is above a given level $c_2$, then a lump sum payment is made which brings the surplus to level $c_1$.\\
Formally speaking, the $(c_1,c_2)$ strategy is defined as follows. Given two constants $c_1,c_2$ such that $c_2 > \beta + c_1$ and $c_1 \geq 0$, we define the sequence of stopping times $\BRA{\tau^{c_1,c_2}_k : k \geq 1}$ such that
\[
\tau^{c_1,c_2}_k := \inf \BRA{t\geq 0: X_t > [(X_0 \vee c_2) + (c_2 - c_1)(k-1)]}, \quad k \geq 1.
\]
The dividend process $(L^{c_1,c_2}_t)_{t\geq 0}$ is defined as
\[
L^{c_1,c_2}_t =\Ind_{\BRA{\tau^{c_1,c_2}_1 < t}}([X_0 \vee c_2] - c_1) + \sum_{k=2}^{\infty}\Ind_{\BRA{\tau^{c_1,c_2}_k<t}}(c_2-c_1), \quad t\geq 0.
\]
Consequently, the controlled process is written as $U^{c_1,c_2}_t := X_t - L^{c_1,c_2}_t$ for all $t \geq 0$. Note that the sequence of stopping times can be expressed in terms of $U^{c_1,c_2}$ as $\tau^{c_1,c_2}_1 = \inf\BRA{t\geq 0: U^{c_1,c_2}_t > c_2}$ and $\tau^{c_1,c_2}_k = \inf\BRA{t \geq \tau^{c_1,c_2}_{k-1}: U^{c_1,c_2}_t > c_2 }$ for $k \geq 2$.\\
Our goal is to determine the existence of a pair $(c_1^\ast,c_2^\ast)$ such that the associated strategy attains the supremum in \eqref{eq.Value}, i.e., that it is an optimal strategy. Moreover, under Assumption \ref{standing-assumption} that the tail of the L\'evy measure is log-convex, we will find conditions for uniqueness and also conditions to determine whether $c_1^\ast > 0$ or $c_1^\ast = 0$.
\subsection{Performance function of a $(c_1,c_2)$ strategy}
\begin{proposition}\label{prop:perf-barrier}
	Fix $c_2 > \beta + c_1 $ with $c_1\geq 0$. Let $\homega : \reals \to [0,\infty)$ be given by $\homega(x) = q + \omega(x)$, then we have
	\[
	v_{c_1,c_2} (x) =
	\begin{cases}
		\PAR{\frac{c_2-c_1-\beta}{\cH^{\homega}(c_2) - \cH^{\homega}(c_1)}} \cH^{\homega}(x) & \text{if $x \in (-\infty,c_2]$,} \\
		x-c_1-\beta + \PAR{\frac{c_2-c_1-\beta}{\cH^{\homega}(c_2) - \cH^{\homega}(c_1)}} \cH^{\homega}(c_1) & \text{if $x \in [c_2, \infty)$.}
	\end{cases}
	\]
\end{proposition}
\begin{proof}
Consider the admissible $(c_1,c_2)$ strategy, and denote its respective controlled process by $(U^{c_1,c_2}_t)_{t\geq 0}$. Let $x \in (-\infty,c_2)$, then by the strong Markov property we get
\[
v_{c_1,c_2}(x) = \e_x\SBRA{e^{-q\tau_{c_2}^+} ; \tau_{c_2}^+ < \sigma_{\omega}}v_{c_1,c_2}(c_2),
\]
where $\tau_{c_2}^+ := \inf\BRA{t\geq 0: U^{c_1,c_2}_t > c_2}$. Note that we have the equality of events $\BRA{ \tau_{c_2}^+ < \sigma_{\omega} } = \BRA{ \int_0^{\tau_{c_2}^+} \omega(U^{c_1,c_2}_t) \diff t < \mathbf{e}_1}$. Integrating w.r.t. the law of $\mathbf{e}_1$ we get
\begin{align*}
	v_{c_1,c_2}(x) &= \e_x\SBRA{ \exp\BRA{-\int_0^{\tau_{c_2}^+} \homega(U^{c_1,c_2}_t) \diff t} ; \tau_{c_2}^+ < \infty }v_{c_1,c_2}(c_2)\\
	&= \frac{\cH^{\homega}(x)}{\cH^{\homega}(c_2)}v_{c_1,c_2}(c_2),
\end{align*}
where in the second equality we have used Theorem 2.5 from \cite{li-palmowski_2018}.\\
On the other hand, for $x \in [c_2,\infty)$ by the definition of the $(c_1,c_2)$ strategy there is a lump sum payment of size $x-c_1$ at time 0 to bring the reserves to level $c_1$, then we have
\begin{align*}
v_{c_1,c_2}(x) &= x - c_1 -\beta + v_{c_1,c_2}(c_1)\\
&= x - c_1 - \beta + \frac{\cH^{\homega}(c_1)}{\cH^{\homega}(c_2)}v_{c_1,c_2}(c_2).
\end{align*}
In particular, for $x = c_2$ we have
\[
v_{c_1,c_2}(c_2) = c_2-c_1-\beta + \frac{\cH^{\homega}(c_1)}{\cH^{\homega}(c_2)}v_{c_1,c_2}(c_2).
\]
Solving for $v_{c_1,c_2}(c_2)$ yields the desired expression.
\end{proof}
\subsection{A candidate optimal strategy}
Equipped with the representation of $v_{c_1,c_2}$ in terms of the generalised scale function, our next goal is to find values $(c^\ast_1,c^\ast_2)$ that will give the \textit{best} strategy, i.e., the one that maximises across all admissible impulse policies and attains the supremum at \eqref{eq.Value}. Given the expression of the performance function, our initial guess will be to find the values $(c^\ast_1,c^\ast_2)$ that minimise
\[
g(c_1,c_2) = \frac{\cH^{\homega}(c_2) - \cH^{\homega}(c_1)}{c_2-c_1-\beta},
\]
where the domain of $g$ is given by $\text{dom}(g) = \BRA{(c_1,c_2): c_1 \geq 0, \, c_2 > \beta + c_1}$.\\
Let $C^\ast$ be the set of minimisers of $g$, i.e.
\[
C^\ast = \BRA{(c_1^\ast, c_2^\ast) \in \text{dom}(g): g(c_1^\ast, c_2^\ast) = \inf_{(c_1,c_2)\in \text{dom}(g)}g(c_1,c_2)}.
\]
\begin{proposition}
	The set $C^\ast$ is non-empty and for each $(c_1^\ast, c_2^\ast) \in C^\ast$ we have
	\[
	\cH^{\homega\prime}(c_2^\ast) = \frac{\cH^{\homega}(c_2^\ast) - \cH^{\homega}(c_1^\ast)}{c_2^\ast-c_1^\ast-\beta}.
	\]
	In addition, one of the following holds:
	\begin{enumerate}
		\item $\cH^{\homega\prime}(c_2^\ast) = \cH^{\homega\prime}(c_1^\ast)$,
		\item $c_1^\ast = 0$.
	\end{enumerate}
\end{proposition}
\begin{proof}
	Let $(c_1,c_2) \in \text{dom}(g)$, then by the mean value theorem for some $\xi \in [c_1,c_2]$ we have
	\begin{align*}
	g(c_1,c_2) &= \PAR{\frac{c_2-c_1}{c_2-c_1-\beta}} \cH^{\homega\prime}(\xi)\\
	&\geq \PAR{\frac{c_2-c_1}{c_2-c_1-\beta}} \min_{\xi \in [c_1,c_2]} \cH^{\homega\prime}(\xi)\\
	& > \min_{\xi \in [c_1,c_2]} \cH^{\homega\prime}(\xi).
	\end{align*}
	From equation (17) in \cite{mata-renaud_2024}we have that
	\[
	\cH^{\homega\prime}(\xi) = Z_q'(\xi-a;\Phi(\phi)) + \int_a^0 W_q'(\xi-y)\omega(y) \cH^{\homega}(y) \diff y.
	\] 
	It follows easily that $\cH^{\homega\prime}(\xi) \to +\infty$ as $\xi \to \infty$, which implies that the minimum of $g$ is not attained when $c_1 \to \infty$. This further implies that there exists $C_1 > 0$ such that
	\[
	\inf_{\text{dom}(g)} g(c_1,c_2) = \inf_{\text{dom}(g), c_1\leq C_1}g(c_1,c_2).
	\]
	Similarly, we show that the minimum is not attained when $c_2 \to \infty$.
	\begin{align*}
		\inf_{c_1 \in [0,C_1]}g(c_1,c_2) &= \inf_{c_1 \in [0,C_1]} \frac{\cH^{\homega}(c_2)}{c_2-c_1-\beta} - \frac{\cH^{\homega}(c_1)}{c_2-c_1-\beta} \\
		& \geq \frac{\cH^{\homega}(c_2)}{c_2-\beta} - \frac{\cH^{\homega}(C_1)}{c_2-C_1-\beta},
	\end{align*}
	it follows that the right-hand side of the inequality goes to $+\infty$ as $c_2 \to \infty$. Hence, the minimum of $g$ is not attained when $c_2 \to \infty$.\\
	Finally, we show that the minimum is not attained when $(c_1,c_2)$ converges to the line $c_2 = c_1 + \beta$. By an application of the mean value theorem we have
	\begin{align*}
		g(c_1,c_2) &\geq \PAR{\frac{c_2-c_1}{c_2-c_1-\beta}} \min_{\xi \in [0,C_1]}\cH^{\homega\prime}(\xi)\\
		&\geq \PAR{\frac{\beta}{c_2-c_1-\beta}} \min_{\xi \in [0,C_1]}\cH^{\homega\prime}(\xi),
	\end{align*}
	where $\min_{\xi \in [0,C_1]}\cH^{\homega\prime}(\xi)>0$. It follows that the right-hand side goes to $+\infty$ as $(c_1,c_2)$ converge to the line $c_2 = c_1 + \beta$, hence the minimum is not attained.\\
	The previous analysis allows us to conclude that the minimiser $(c_1^\ast,c_2^\ast)$ is either an interior point of $\text{dom}(g)$ or is such that $c_1^\ast = 0$.\\
	In case that $(c_1^\ast,c_2^\ast)$ is an interior point, then $g$ is differentiable and satisfies
	\[
	\frac{\partial g}{\partial c_1}(c_1^\ast,c_2^\ast) = 0, \quad\text{and,}\quad \frac{\partial g}{\partial c_2}(c_1^\ast,c_2^\ast)=0.
	\]
	After some algebraic manipulations, we observe that
	\[
	\cH^{\homega\prime}(c_2^\ast) = \frac{\cH^{\homega}(c_2^\ast) - \cH^{\homega}(c_1^\ast)}{c_2^\ast-c_1^\ast-\beta},
	\]
	hence item (1) follows.\\
	In the case where $c_1^\ast=0$, then $c_2^\ast$ minimises the function $g_0(c_2):= g(0,c_2)$. It follows that $g_0'(c_2^\ast)=0$, hence we also have $\cH^{\homega\prime}(c_2^\ast) = \frac{\cH^{\homega}(c_2^\ast) - \cH^{\homega}(0)}{c_2^\ast-\beta}$.
\end{proof}
\begin{corollary}\label{cor_definetti}
	For each $(c_1^\ast,c_2^\ast)\in C^\ast$ we have
	\[
	v_{c_1^\ast,c_2^\ast}(x) = \begin{cases}
		 \frac{\cH^{\homega}(x)}{\cH^{\homega\prime}(c_2^\ast)} & \text{if $x \in (-\infty,c_2^\ast]$,} \\
		x-c_2^\ast + \frac{\cH^{\homega}(c_2^\ast)}{\cH^{\homega\prime}(c_2^\ast)} & \text{if $x \in [c_2^\ast, \infty)$.}
	\end{cases}
	\]
\end{corollary}
Note that this representation coincides with the performance function of a barrier strategy at level $c_2^\ast$ in De Finetti's problem under state dependent ruin, as has been studied in \cite{mata-renaud_2024}. This representation will be useful in the proofs of optimality.
\section{Verification of Optimality}\label{Sec:Verification}
We prove that the performance function of a $(c_1^\ast,c_2^\ast)$ strategy is optimal for any $(c_1^\ast,c_2^\ast) \in C^\ast$.\\
To prove optimality, we need the following verification lemma which states the necessary conditions for a strategy to be optimal.
\begin{lemma}[Verification Lemma]
	Suppose that $\hat{\pi}$ is an admissible impulse control strategy such that $v_{\hat{\pi}}$ is sufficiently smooth on $\reals$, and satisfies
	\begin{align}
		(\Gamma - \homega(\cdot))v_{\hat{\pi}}(x) &\leq 0, &\text{for almost all } x \in \reals;\label{variational:1}\\
		v_{\hat{\pi}}(x) - v_{\hat{\pi}}(y ) & \geq x-y-\beta, & x\geq y.\label{variational:2}
	\end{align}
	Then $v_{\hat{\pi}} = v_{\ast}$ for almost every $x\in \reals$ and hence $\hat{\pi}$ is an optimal strategy.
\end{lemma}
\begin{proof}
	Throughout this proof we write $h = v_{\hat{\pi}}$. Let $\pi$ be an admissible strategy, and $U$ its associated controlled process. Since $h$ is sufficiently smooth, an application of Meyer-It\^{o}'s formula (see Theorem 70 in \cite{protter_2004}) to $Y_t h(U_t)$, with $Y_t := \exp\BRA{\int_0^t \homega(U_s)\diff s}$, gives:
	\begin{align*}
		Y_t h(U_t) - h (U_0) &= \int_{0+}^t Y_s\left( \Gamma h (U_s) - (\omega(U_s)+q) h (U_s) \right) \diff s - \int_{0}^t Y_s h^{\prime} (U_s) \diff L^c_s \\
		& \qquad - \sum_{0 \leq s < t} Y_s \left\lbrace h (X_s-L_{s-}) - h (X_s - L_{s-} - \Delta L_s) \right\rbrace + \int_{0+}^t Y_s \diff M_s ,
	\end{align*}
	where $t \mapsto \int_{0+}^t Y_s \diff M_s$ is a local martingale.\\
	Recall that admissible strategies are pure jump processes and $\Delta L_t \geq \beta$ for all $t\geq 0$, then applying inequality \eqref{variational:2}, we have
	\begin{align*}
		Y_t h(U_t) - h (U_0) &\leq \int_{0+}^t Y_s\left( \Gamma h (U_s) - (\omega(U_s)+q) h (U_s) \right) \diff s - \sum_{0 \leq s < t} Y_s(\Delta L_s - \beta) + \int_{0+}^t Y_s \diff M_s.
	\end{align*}
	In addition, since for almost all $x \in \reals \setminus \{a_1,\dots,a_n,0\}$ we have the variational inequality \eqref{variational:1} then we can further write
	\[
	Y_t h(U_t) - h (U_0) \leq - \sum_{0 \leq s < t} Y_s(\Delta L_s - \beta) + \int_{0+}^t Y_s \diff M_s.
	\]
	Now, consider $(T_n)_{n\geq 1}$ a localising sequence. Applying the last inequality at $t=T_n$ and taking expectations, we can write
	\[
	h(x) \geq \e_x\SBRA{Y_{T_n} h(U_{T_n})} + \e_x\SBRA{ \sum_{0 \leq s < T_n} Y_s(\Delta L_s - \beta) }.
	\]
	Finally, by taking the limit as $n \to \infty$, we obtain
	\[
	h(x) \geq \e_x\SBRA{ \sum_{0 \leq s < \infty} Y_s(\Delta L_s - \beta) } = v_\pi(x),
	\]
	where in the last equality we have used Lemma \ref{lem:alternate}. Since $\pi$ is arbitrary, the result follows.
\end{proof}
\begin{lemma}\label{lem:HJB:1}
	Let $(c_1^\ast,c_2^\ast) \in C^\ast$. Then for $x \geq y \geq 0$,
	\begin{equation}\label{eq:lin:ineq}
		v_{c_1^\ast,c_2^\ast}(x) - v_{c_1^\ast,c_2^\ast}(y) \geq x - y -\beta.
	\end{equation}
\end{lemma}
\begin{proof}
	Firstly, we note that $v_{c_1^\ast,c_2^\ast}$ is an increasing function, so throughout this proof we can assume without loss of generality that $x - y > \beta$.\\
	Now, we consider the case when $x \geq y \geq c_2^\ast$, then it follows easily that $v_{c_1^\ast,c_2^\ast}(x) - v_{c_1^\ast,c_2^\ast}(y) = x-y \geq x-y-\beta$. Second, in the case where $c_2^\ast \geq x \geq y$, using Corollary \ref{cor_definetti} we have
	\[
	v_{c_1^\ast,c_2^\ast}(x) - v_{c_1^\ast,c_2^\ast}(y) = \frac{\cH^{\homega}(x) - \cH^{\homega}(y)}{\cH^{\homega\prime}(c_2^\ast)}.
	\]
	Recall that $(c_1^\ast,c_2^\ast) \in C^\ast$, so that we have the equality
	\begin{align*}
		\cH^{\homega\prime}(c_2^\ast) &= \frac{\cH^{\homega}(c_2^\ast) - \cH^{\homega}(c_1^\ast)}{c_2^\ast - c_1^\ast - \beta}\\
		&\leq \frac{\cH^{\homega}(x) - \cH^{\homega}(y)}{x - y - \beta},
	\end{align*}
	where the inequality follows simply from the fact that $(c_1^\ast,c_2^\ast) \in C^\ast$. Plugging this back into $v_{c_1^\ast,c_2^\ast}(x) - v_{c_1^\ast,c_2^\ast}(y)$ gives
	\[
	v_{c_1^\ast,c_2^\ast}(x) - v_{c_1^\ast,c_2^\ast}(y) \geq x - y -\beta.
	\]
	Finally, if $x \geq c_2^\ast \geq y$, then by Corollary \ref{cor_definetti}
	\begin{align*}
		v_{c_1^\ast,c_2^\ast}(x) - v_{c_1^\ast,c_2^\ast}(y) &= x - c_2^\ast + \frac{\cH^{\homega}(c_2^\ast)}{\cH^{\homega\prime}(c_2^\ast)} - \frac{\cH^{\homega}(y)}{\cH^{\homega\prime}(c_2^\ast)} \\
		& \geq x - c_2^\ast + c_2^\ast - y - \beta \\ &= x-y-\beta,
	\end{align*}
	where in the inequality we have used once again that $(c_1^\ast,c_2^\ast) \in C^\ast$.
\end{proof}
We are ready to state the main result of this section.
\begin{theorem}
	Suppose that there exists $(c_1^\ast,c_2^\ast) \in C^\ast$ such that
	\[
	\cH^{\homega\prime}(a) \leq \cH^{\homega\prime}(b), \quad \text{for all } c_2^\ast \leq a \leq b.
	\]
	Then the strategy $(c_1^\ast,c_2^\ast)$ is an optimal strategy for the impulse control problem.
\end{theorem}
\begin{proof}
	Let $(c_1^\ast,c_2^\ast) \in C^\ast$ that satisfies the conditions of the statement. Firstly, thanks to Lemma \ref{lem:HJB:1} we verify that $(c_1^\ast,c_2^\ast)$ satisfies the variational inequality \eqref{variational:2}.\\
	Regarding the variational inequality \eqref{variational:1}, in the proof of Lemma 4.1 in \cite{mata-renaud_2024} it was shown that
	\[
	(\Gamma - \homega(\cdot)) \cH^{\homega}(x) = 0, \quad x \in \reals\setminus\BRA{a_1,\cdots,a_n,a_{n+1}},
	\]
	then it follows that
	\[
	(\Gamma - \homega(\cdot)) v_{c_1^\ast,c_2^\ast}(x) = 0, \quad x \in (-\infty,c_2^\ast)\setminus\BRA{a_1,\cdots,a_n,a_{n+1}}.
	\]
	On the other hand, using that $\cH^{\homega\prime}(c_2^\ast) \leq \cH^{\homega\prime}(a) \leq \cH^{\homega\prime}(b)$, then we can proceed verbatim as in the proof of Theorem 2 in \cite{loeffen_2008} to deduce
	\[
	(\Gamma - \homega(\cdot)) v_{c_1^\ast,c_2^\ast}(x) \leq 0, \quad x \in (c_2^\ast,\infty).
	\]
	We conclude that $(c_1^\ast,c_2^\ast)$ is an optimal strategy.
\end{proof}
\section{Optimal strategies when the tail of the L\'evy measure is log-convex}\label{Sec:Optim:Convex}
The results presented in the preceding sections hold for a general class of SNLP's, however, under our standing assumption that the tail of L\'evy measure is log-convex, we are able to further characterise the optimal barrier strategy. In this section, we will prove that $C^\ast$ consists of one element, and we will provide conditions to determine whether $c_1^\ast = 0$ or $c_1^\ast > 0$, as well as the characterisation of $c^\ast_2$.\\
Under our standing Assumption \ref{standing-assumption}, we have that $W_q$ is $C^1$ on $(0,\infty)$ and moreover, Proposition 2.1 in \cite{mata-renaud_2024} gives that $x \mapsto \cH^{\homega\prime}(x)$ is log-convex on $(0,\infty)$, hence it is also convex. In addition, the log-convexity of $\cH^{\homega\prime}$ implies that there exists $a^\ast \in [0,\infty)$ such that
\[
\cH^{\homega\prime}(a^\ast) \leq \cH^{\homega\prime}(x), \quad \text{for all } x\geq 0.
\]
It also follows that $\cH^{\homega\prime}$ is strictly decreasing on $(0,a^\ast)$ and strictly increasing on $(a^\ast,\infty)$. Throughout this section we will use the following function $\zeta : (0,a^\ast) \to (a^\ast,\infty)$, which is implicitly defined as the one that satisfies $\cH^{\homega\prime}(x) = \cH^{\homega\prime}(\zeta(x))$ for all $x \in (0,a^\ast)$.
\begin{remark}[Properties of $\zeta$]
	By its definition, $\zeta$ is strictly decreasing. In addition, using that $\cH^{\homega\prime}$ is convex on $(0,\infty)$, hence left- and right-differentiable, we obtain for $a \in (0,a^\ast)$
	\begin{align*}
		\lim_{x \downarrow a} \frac{\zeta(x)-\zeta(a)}{x-a} &= \lim_{x\downarrow a}\frac{(\zeta(x)-\zeta(a))(\cH^{\homega\prime}(x) - \cH^{\homega\prime}(a))}{(x-a)(\cH^{\homega\prime}(\zeta(x)) - \cH^{\homega\prime}(\zeta(a)))}\\
		&= \PAR{\lim_{y \uparrow \zeta(a)}\frac{y - \zeta(a)}{\cH^{\homega\prime}(y) - \cH^{\homega\prime}(\zeta(a))} }\PAR{\lim_{x\downarrow a}\frac{ \cH^{\homega\prime}(x) - \cH^{\homega\prime}(a) }{x-a}}\\
		&= \frac{\cH^{\homega\prime\prime}_+(a)}{\cH^{\homega\prime\prime}_-(\zeta(a))},
	\end{align*}
	hence $\zeta$ is right-differentiable. By a similar computation, we also obtain that $\zeta$ is left-differentiable.\\
	Due to convexity of $\cH^{\homega\prime}$ and the definition of $a^\ast$, we have that $\cH^{\homega\prime\prime}_+(x)<0$ for all $x \in (0,a^\ast)$, and that $\cH^{\homega\prime\prime}_-(x)>0$ for $x \in (a^\ast,\infty)$. It then follows that $\zeta'_+(x) < 0$ for all $x \in (0,a^\ast)$.\\
\end{remark}
We state our main result.
\begin{theorem}\label{thm.unique}
	If the L\'evy measure has a log-convex tail, then there exists a unique $(c^\ast_1,c^\ast_2)$ strategy that is optimal for the impulse control problem. Moreover, $c_1^\ast = 0$ if and only if $\beta \geq \beta_{\max}$ or $a^\ast = 0$, where $\beta_{\max}$ is given by 
	\begin{equation}\label{eq:betamax}
		\beta_{\max}:= \zeta(0) - \frac{\cH^{\homega}(\zeta(0)) - \cH^{\homega}(0)}{\cH^{\homega\prime}(0+)}.
	\end{equation}
\end{theorem}
The proof of this Theorem is rather long, so we will break it down into a series of lemmas.
\begin{lemma}\label{lemma.unique} The set $C^\ast$ consists of only one element, hence there exists a unique impulse control strategy $(c_1^\ast,c_2^\ast)$ that is optimal.
\end{lemma}
\begin{proof}
	Suppose there exist $(c_1,c_2), (\hat{c}_1,\hat{c}_2) \in C^\ast$. Since they are minimisers of $g$, we have that $\cH^{\homega\prime}(c_2)= \cH^{\homega\prime}(\hat{c}_2)$. However, since both $c_2,\hat{c_2} > a^\ast$ and $\cH^{\homega\prime}$ is strictly increasing on $(a^\ast,\infty)$, then we must have $c_2 = \hat{c}_2$.\\
	By a similar reasoning, if both $c_1$ and $\hat{c}_1$ are strictly positive, and since $c_1,\hat{c}_1 < a^\ast$ and $\cH^{\homega\prime}$ is strictly decreasing on $(0,a^\ast)$, then we must have $c_1 = \hat{c}_1$.\\
	Finally, suppose that $c_1 = 0$ and $\hat{c}_1 > 0$. On the one hand we have
	\[
	\cH^{\homega\prime}(c_1) = \frac{\cH^{\homega}(c_2)-\cH^{\homega}(c_1)}{c_2-c_1-\beta}.
	\]
	On the other hand, we have
	\[
	\cH^{\homega\prime}(0) = \frac{\cH^{\homega}(c_2)-\cH^{\homega}(0)}{c_2-\beta}.
	\]
	Since they are both minimisers of $g$, we also have $ \cH^{\homega\prime}(c_1) = \cH^{\homega\prime}(0) $, which after some algebraic manipulations gives
	\begin{align*}
	\cH^{\homega\prime}(c_1) &= \frac{\cH^{\homega\prime}(c_1)(c_2-\beta) + \cH^{\homega}(0) - \cH^{\homega}(c_1)}{c_2-c_1-\beta}\\
	&= \frac{\cH^{\homega\prime}(c_1)(c_2-\beta) - \cH^{\homega\prime}(\xi)c_1}{c_2-c_1-\beta},
	\end{align*}
	where in the second equality we have applied the mean value theorem, hence $\xi \in (0,c_1)$. However, since $\cH^{\homega\prime}$ is strictly decreasing on $(0,c_1)$, we obtain
	\begin{align*}
	\cH^{\homega\prime}(c_1) &= \frac{\cH^{\homega\prime}(c_1)(c_2-\beta) - \cH^{\homega\prime}(\xi)c_1}{c_2-c_1-\beta}\\
	&< \cH^{\homega\prime}(c_1),
	\end{align*}
	which is a contradiction. It follows that $c_1$ has to be equal to $\hat{c}_1$.
\end{proof}
In the ext result, we provide a way to compute $(c_1^\ast,c_2^\ast)$. To this end, we will introduce some auxiliary functions that will allow us to reduce the 2-dimensional minimisation into two auxiliary 1-dimensional minimisation problems. Let
\[
c_{1\max}:= \inf\BRA{c_1\in (0,a^\ast): \zeta(c_1)-c_1 \leq \beta},
\]
where we set $c_{1\max} = 0$ when $\lim_{x\downarrow 0}\zeta(x)\leq \beta$. Define the functions $g_1 : (0,c_{1\max})\to (0,\infty)$ and $g_0:(\beta,\infty)\to (0,\infty)$ as
\begin{align}
	g_1(c_1) &= \frac{\cH^{\homega}(\zeta(c_1)) - \cH^{\homega}(c_1)}{\zeta(c_1)-c_1-\beta}, \label{eq.g1}\\
	g_0(c_2) &= \frac{\cH^{\homega}(c_2) - \cH^{\homega}(0)}{c_2-\beta}.\label{eq.c2}
\end{align}
\begin{lemma}\label{lemma.g0-min} The function $g_0$ has a unique minimiser.
\end{lemma}
\begin{proof}
		By differentiating $g_0$ we obtain:
	\begin{align*}
		g'_{0}(x)&=\frac{1}{x-\beta}(\cH^{\homega\prime}(x) - g_0(x)).
	\end{align*}
	From these expressions we deduce the following:
	\begin{enumerate}
		\item $g_0$ is strictly decreasing if and only if $ g_0 < \cH^{\homega\prime}$;
		\item $g_0$ is strictly increasing if and only if $ g_0 > \cH^{\homega\prime}$;
		\item $g_0$ has a critical point at $x$ if and only if $ g_0(x) = \cH^{\homega\prime}(x)$.
	\end{enumerate}
	In addition, we show that $g_0$ has a unique minimiser. On the one hand, we have that $\lim_{x\downarrow \beta}g_0(x) = + \infty$ whereas $\cH^{\omega\prime}(\beta)< \infty$, hence it follows that there exists an open vicinity around $\beta$ such that $g_0 > \cH^{\homega\prime}$. On the other hand, for $\varepsilon \in (0,1)$, and $x$ large enough we write 
	\[
	g_0(x) = \PAR{\frac{\cH^{\homega}(x) - \cH^{\homega}(0)}{(1-\varepsilon)x}}\PAR{\frac{(1-\varepsilon)x}{x-\beta}}.
	\]
	Note that if $x$ is large enough, then $\frac{(1-\varepsilon)x}{x-\beta}\leq 1$. In addition, the mean value theorem implies that there exists some $\xi \in [0,x]$ such that $\cH^{\homega\prime}(\xi) = \frac{\cH^{\homega}(x) - \cH^{\homega}(0)}{x}$. Moreover, since $\cH^{\homega\prime}$ is increasing on $(a^\ast,\infty)$ then for $x \gg a^\ast$ we have
	\[
	g_0(x) < \frac{\cH^{\homega\prime}(x)}{1-\varepsilon},
	\]
	where it follows that $g_0(x) < \cH^{\homega\prime}(x)$ for $x$ sufficiently large. This analysis implies that there exists a unique $\hat{c} \in (\beta \vee a^\ast,\infty)$ such that $g_0$ is decreasing on $(\beta,\hat{c})$, increasing on $(\hat{c},\infty)$, and $g_0(\hat{c}) = \cH^{\homega\prime}(\hat{c})$.
\end{proof}
Finally, we derive a condition to determine whether $c_1^\ast > 0 $ or $c_1^\ast=0$.
\begin{lemma}\label{lemma.condition}
	Recall the parameter $\beta_{\max}$ defined in \eqref{eq:betamax}. Consider the following cases:
	\begin{enumerate}
		\item $a^\ast > 0$ and $\beta < \beta_{\max}$;
		\item $a^\ast > 0$ and $\beta \geq \beta_{\max}$;
		\item $a^\ast = 0$.
	\end{enumerate}
	In case (1) we have that $c_1^\ast > 0$; whereas in cases (2) and (3) we have that $c_1^\ast=0$.
\end{lemma}
\begin{proof}
	In case (1), we will show that $c_1^\ast > 0$. First, by differentiating $g_1$ we obtain
		\begin{align*}
		g'_{1+}(x)&=\frac{\zeta'_+(x)-1}{\zeta(x)-x-\beta}(\cH^{\homega\prime}(x) - g_1(x)),
		\end{align*}
		then we deduce the following:
			\begin{enumerate}
					\item $g_1$ is strictly decreasing if and only if $ g_1 < \cH^{\homega\prime}$;
					\item $g_1$ is strictly increasing if and only if $ g_1 > \cH^{\homega\prime}$;
					\item $g_1$ has a critical point at $x$ if and only if $ g_1(x) = \cH^{\homega\prime}(x)$.
				\end{enumerate}
	Given this behaviour, we now show that $g_1$ has a unique minimiser. Indeed, on the one hand we have
	\begin{align*}
		g_1(0) &= \frac{\cH^{\homega}(\zeta(0)) - \cH^{\homega}(0)}{\zeta(0)-\beta}\\
		& < \frac{\cH^{\homega}(\zeta(0)) - \cH^{\homega}(0)}{\zeta(0)-\beta_{\max}}\\
		&= \cH^{\homega\prime}(0+),
	\end{align*}
	where in the last equality we have used the definition of $\beta_{\max}$. On the other hand, we have
	\[
	\lim_{x \uparrow c_{1\max}}g_1(x) = +\infty > \cH^{\homega\prime}(c_{1\max}).
	\]
	This analysis, together with the fact that $\cH^{\homega\prime}$ is strictly decreasing on $(0,a^\ast)$ implies that there exists a unique $\bar{c} \in (0,c_{1\max})$ such that $g_1$ is decreasing on $(0,\bar{c})$, increasing on $(\bar{c},c_{1\max})$, and that $g_1(\bar{c})= \cH^{\homega\prime}(\bar{c})$.\\
	Now we turn to prove that $(c_1^\ast,c_2^\ast) = (\bar{c},\zeta(\bar{c}))$. By the definition of $\bar{c}$ and the fact that $\cH^{\homega\prime}$ is decreasing on $(0,a^\ast)$ we have $g_1(\bar{c})< \cH^{\homega\prime}(0+)$. Recall the unique minimiser of $g_0$, denoted by $\hat{c}$, which we obtained in Lemma \ref{lemma.g0-min}. It follows that, if $g_0(\hat{c})\geq \cH^{\homega\prime}(0+)$, then $g_1(\bar{c}) < g_0(\hat{c})$ and so $g$ is minimised at $(\bar{c},\zeta(\bar{c}))$. Conversely, if $g_0(\hat{c}) < \cH^{\homega\prime}(0+)$, then we get
	\[
	\lim_{x\downarrow 0}\frac{\partial}{\partial x} g(x,\hat{c}) = \frac{g_0(\hat{c}) - \cH^{\homega\prime}(0+)}{\hat{c}-\beta} < 0,
	\]
	hence $(0,\hat{c})$ is not a minimiser of $g$. We conclude that $(c_1^\ast,c_2^\ast) = (\bar{c},\zeta(\bar{c}))$.\\
	In case (2), first we notice that $g_1(0) \geq \cH^{\homega\prime}(0+)$. Now, for $x \in (0,c_{1\max})$ and using the mean value theorem we have
	\begin{align*}
		g_1(x) &= \frac{\cH^{\homega}(\zeta(x)) - \cH^{\homega}(x)}{\zeta(x) - x -\beta}\\
		&=\frac{\zeta(x) - x}{\zeta(x) - x -\beta}\cH^{\homega\prime}(\xi)\\
		&> \cH^{\homega\prime}(x),
	\end{align*}
	where in the inequality we have used that $\cH^{\homega\prime}$ is decreasing on $(0,c_{1\max})$ and that $\frac{\zeta(x) - x}{\zeta(x) - x -\beta} \geq 1$. It follows that $g_1$ is strictly increasing on $(0,c_{1\max})$. Hence $c_1^{\ast} = 0$, and even further $(c_1^\ast,c_2^\ast) = (0,\hat{c})$.\\
	Finally, in case (3), first note that $c_{1\max} = 0$ and that $\cH^{\homega\prime}$ is strictly increasing on $(0,\infty)$. It now follows that  $c_1^{\ast} = 0$, and even further $(c_1^\ast,c_2^\ast) = (0,\hat{c})$.
\end{proof}
To sum up, using Lemmas \eqref{lemma.unique}--\eqref{lemma.condition} we have obtained that the optimal barriers $(c_1^\ast,c_2^\ast)$ are unique, and we have derived conditions to characterise and compute them in terms of the cost parameter $\beta$ and the generalised scale function. Hence, the proof of Theorem \ref{thm.unique} is complete.
\section{Numerical Experiments}\label{Sec:Numerical}
We provide a numerical illustration of our results. Throughout this section, we assume that our L\'evy process $X$ is given by
\[
X_t = x + \mu t + \sigma B_t - \sum_{i = 1}^{N_t} Y_i,
\]
where $\mu = 0.075, \, \sigma = 0.5, \, B=(B_t)_{t\geq 0}$ is a standard Brownian motion, $(N_t)_{t\geq 0}$ is an homogeneous Poisson process with arrival rate $\lambda = 0.5$ independent of $B$, and $(Y_i)_{i \geq 1}$ is a sequence of iid exponential random variables with parameter $\alpha = 9$; finally, we let $q = 0.025$. In this model, the scale function $W_q$ can be computed explicitly as 
\[
W_q(x) = \sum_{i=1}^3 D_i \mathrm{e}^{\theta_i x},
\]
where $\theta_i, \, i = 1,2,3$ are the distinct roots of $s \mapsto \psi(s) - q$ satisfying $\theta_1 > 0$ and $\theta_2,\, \theta_3 < 0$, and  $D_i $ are given by $D_i = \frac{1}{\psi'(\theta_i)}$. Consequently, the second scale function $Z_q$ can also be found explicitly by standard integration.\\
For this numerical illustration, we consider that the bankruptcy rate function is of the form $\omega(x) = \phi \Ind_{(-\infty,a]} + (\phi + m(x-a))\Ind_{(a,0]}$, with $a=-1$, $\phi = 1.5$, and $m = -0.15$. Note that, thanks to the explicit expression for the scale functions $W_q$ and $Z_q$, we are able to compute numerically the generalised scale function $\cH^{\homega}$ by the iteration scheme described in Section 2.2 in \cite{mata-renaud_2024}.\\
Using these parameter values, first we find that $\beta_{\max}= 0.009$, hence for our numerical experiment we take $\beta = 0.001$. In Figure \ref{Fig:Step:Derivative} we have plotted $\cH^{\homega\prime}$ (Blue and yellow curves), as well as the mapping $g_1$ (Red curve). We observe that $c_1^\ast > 0$ is the unique minimiser of $g_0$, and it is also the unique intersection point of $\cH^{\homega\prime}$ and $g_1$. Afterwards we find $c_2^\ast$ as the one satisfying $\cH^{\homega\prime}(c_1^\ast) = \cH^{\homega\prime}(c_2^\ast)$. Equipped with the optimal barriers $(c_1^\ast,c_2^\ast)$ we are able to compute the value function as well as its first derivative, we plot them in Figure \ref{Fig:Omega:Step}. It is worth noting that the $C^1$ fit condition at $c_2^\ast$ is satisfied, however the value function is not globally $C^2$, which is expected due to the representation from Corollary \ref{cor_definetti}. 
\begin{figure}[h!]
	\centering
	\includegraphics[scale=0.75]{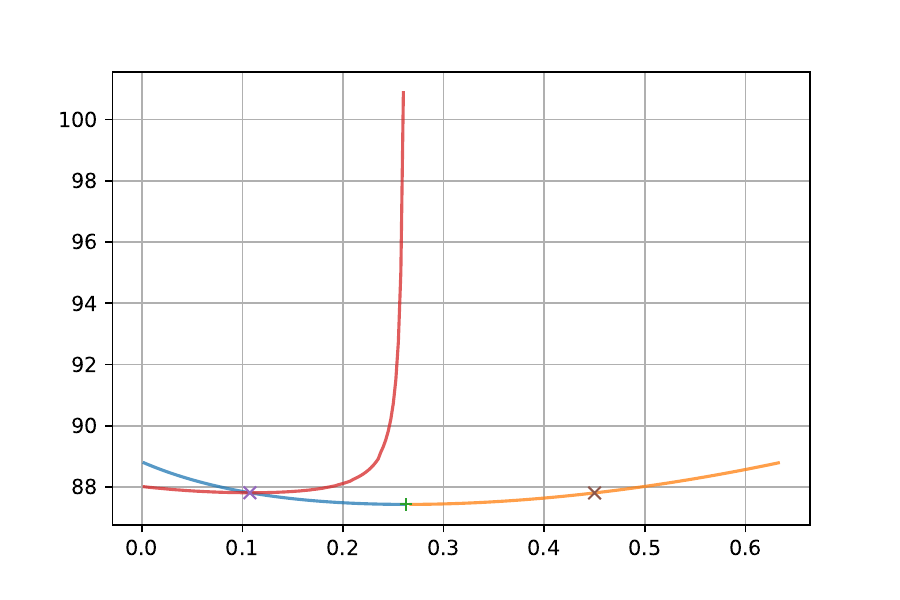}
	\caption{Plot of the function $g_0$ (Red), and $\cH^{\homega\prime}$ (Blue and Yellow). The unique minimiser of $g_1$, denoted by $c_1^\ast$ as well as $\zeta(c_1^\ast)$ are marked as $\times$. The unique minimiser of $\cH^{\homega\prime}$ is marked as a $+$.}
	\label{Fig:Step:Derivative}
\end{figure}
\begin{figure}[h!]
	\centering
	\begin{tabular}{cc}
	\includegraphics[scale=0.5]{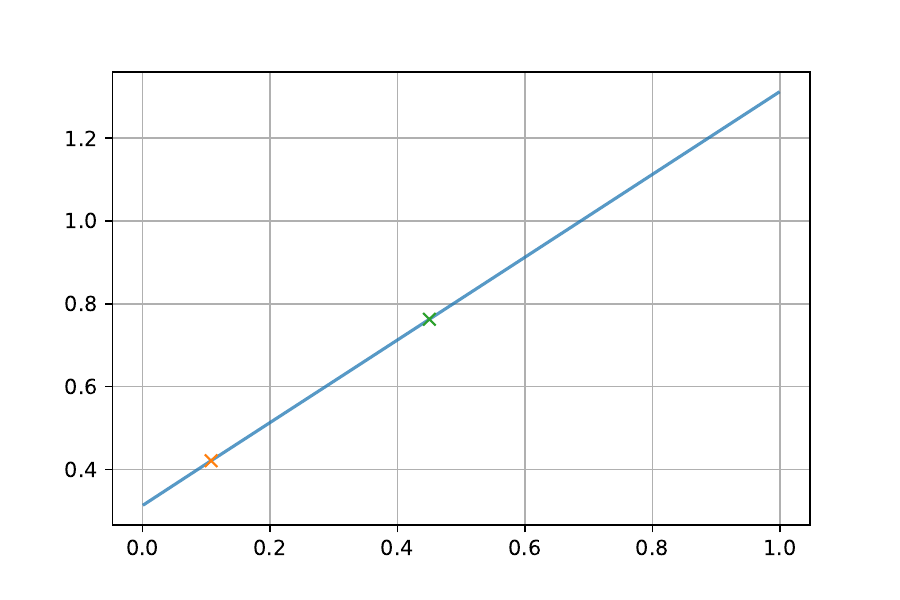} &
	\includegraphics[scale = 0.5]{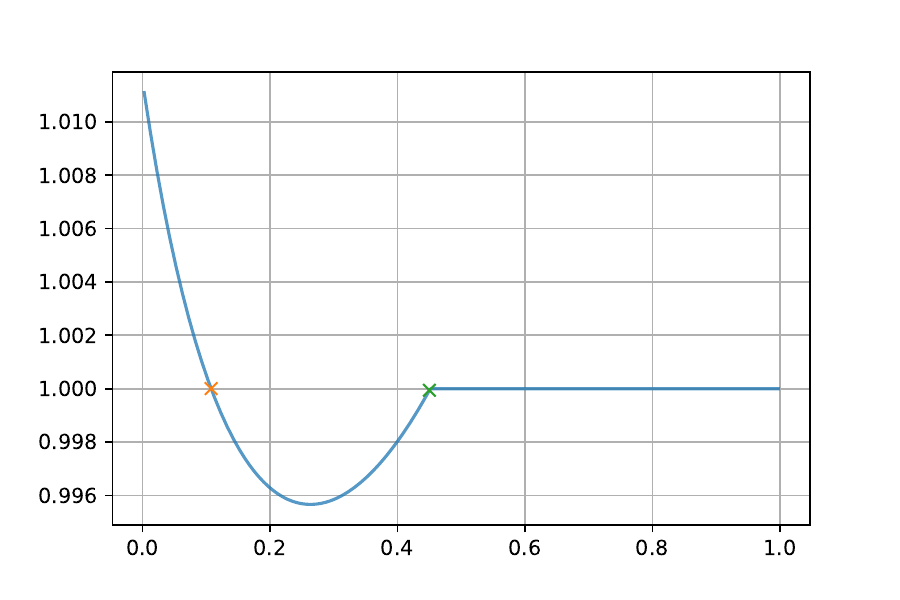}
	\end{tabular}
	\caption{Value function (left) and its first derivative (right) for $\omega$ step functions. The optimal dividend barrier is indicated as a cross. The dashed line corresponds to the value function of Parisian ruin. }
	\label{Fig:Omega:Step}
\end{figure}

\section*{Acknowledgements}

Funding in support of this work was provided by a CRM-ISM Postdoctoral Fellowship from the Centre de recherches mathématiques (CRM) and the Institut des sciences mathématiques (ISM).
 \section*{Competing Interests}
The author has no relevant financial or non-financial interests to disclose.
%
%
\bibliographystyle{abbrv}
\bibliography{references-SNLPs.bib}

\end{document}